\documentclass[a4paper, 12pt]{amsart}
\usepackage{amsmath,amssymb,amsthm}
\usepackage{amscd}
\usepackage{array,enumerate}
\newtheorem{Thm}{Theorem}[section]
\newtheorem{Prop}[Thm]{Proposition}
\newtheorem{Lem}[Thm]{Lemma}
\newtheorem{Cor}[Thm]{Corollary}

\theoremstyle{definition}
\newtheorem{Rem}[Thm]{Remark}

\newtheorem{Def}[Thm]{Definition}

\newcommand{\Cs}{C$^\ast$}
\newcommand{\Ws}{W$^\ast$}
\newcommand{\id}{\mbox{\rm id}}

\newcommand{\rg}{\mathop{{\mathrm C}_{\mathrm r}^\ast}}

\newcommand{\rc}{\mathop{\rtimes _{\mathrm r}}}
\newcommand{\votimes}{\mathop{\bar{\otimes}}}

\newcommand{\vnc}{\mathop{\bar{\rtimes}}}

\DeclareMathOperator{\supp}{supp}

\DeclareMathOperator{\im}{Im}
\DeclareMathOperator{\bigfp}{\lower0.25ex\hbox{\LARGE $\ast$}}

\title[The AP and exactness for locally compact groups]
{The approximation property and exactness of locally compact groups}
\author{Yuhei Suzuki}
\address{Graduate school of mathematics, Nagoya University, Chikusaku, Nagoya, 464-8602, Japan}
\email{yuhei.suzuki@math.nagoya-u.ac.jp}
\begin{document}

\begin{abstract}
We extend a theorem of Haagerup and Kraus in the \Cs-algebra context:~
for a locally compact group with the approximation property (AP),
the reduced \Cs-crossed product construction preserves the strong operator approximation property (SOAP).
In particular their reduced group \Cs-algebras have the SOAP.
Our method also solves another open problem:~ the AP implies exactness for general locally compact groups.
\end{abstract}
\maketitle
\section{Introduction}
Recent research (see e.g., \cite{BCL}, \cite{BDV}, \cite{HR}, \cite{Oza16}, \cite{Suzsim}, \cite{Suzeq}) shows that non-discrete locally compact groups
also provide rich sources of interesting operator algebras.
They also reveal attractive and fruitful interactions between locally compact group theory and theory of operator algebras.
Considering these developments, it is natural and important to attempt extending
the results known only in the discrete case to general locally compact groups.

In this paper, we carry out this idea for the approximation property (AP).
This property is introduced by Haagerup and Kraus in \cite{HK} 
(see also \cite{ER}, \cite{Kra}) as a natural
group and operator space analogue of Grothendieck's Banach space approximation property \cite{Grot}.
Among other things, they show a strong relation between the AP and
the validity of Fubini-type theorems of associated operator algebras.
Indeed, for discrete groups, the AP is characterized by slice map properties
of the reduced group operator algebras.
For general locally compact groups with the AP,
they show that
the \Ws-crossed product construction preserves slice map properties.
(In particular this applies to the group von Neumann algebras.)
See Theorems 3.1 and 3.2 of \cite{HK} for the precise statement.
The proofs of these theorems are, unlike the discrete case, based on hard manipulations
involving unbounded operator-valued weights (\cite{Ha78a}, \cite{Ha78b}, \cite{Ha79a}, \cite{Ha79b}).
Since there are no analogous results in the \Cs-algebra context,
the \Cs-algebra case has remained open.

In this paper, we solve this problem by extending completely bounded multiplier operators to the reduced crossed products in a canonical way.
We then develop a technique to understand their stable pointwise convergence conditions.
This together with mollification arguments gives nice approximations
on the reduced crossed products.
As a byproduct of our approach, we obtain the implication
\[{\rm AP} \Longrightarrow{\rm exactness}\] for general locally compact groups.
This is known for discrete groups \cite{HK},
but the proof does not extend to the general case.
(For instance, for non-discrete groups, it is not known
if the exactness of a group follows from the exactness of the reduced group \Cs-algebra.)
We note that, for second countable weakly amenable groups, this has been proved in a recent paper \cite{BCL} based on metric geometry.
As a result of our theorem, \cite{BCL}, and \cite{CEO}, we obtain
important results on the Baum--Connes conjecture
for locally compact groups with the AP.
See \cite{BCL} for details.

\subsection*{Added in 25th January, 2019.} After this paper has appeared on arXiv,
the author was informed from Jason Crann and Matthias Neufang that the implication `AP $\Rightarrow$ exactness' was independently obtained
in their forthcoming paper \cite{CN}.
\section{Preliminaries}
For general facts on locally compact groups and continuous crossed products, we refer the reader to the books \cite{Fol}, \cite{Ped}.
For general backgrounds and applications of finite dimensional approximations
of groups and operator algebras, see \cite{BO}.
\subsection{Notations}
We first fix some notations used in the paper.

The symbols `$\odot$', `$\otimes$', `$\votimes$' stand for the algebraic,
minimal, and von Neumann algebra tensor product respectively.
The symbol `$\otimes$' is also used for Hilbert spaces.

For a \Cs-algebra $A$, denote by $\mathcal{M}(A)$ the multiplier algebra of $A$.

We put $\mathbb{K}:=\mathbb{K}(\ell^2(\mathbb{N}))$ and $\mathbb{B}:=\mathbb{B}(\ell^2(\mathbb{N}))$ for short.

For a Banach space $E$ and its dual space $E^\ast$,
denote by $\langle , \rangle$ the pairing map $E \times E^\ast \rightarrow \mathbb{C}$.

Let $A$ be a \Cs-algebra.
Let $X$ be a locally compact (Hausdorff) space equipped with a Radon measure $\mu$.
For $a\in C_c(X, A)$,
define
\[\|a\|_1:= \int_X \| a(x)\| d\mu(x),~ \|a\|_\infty:= \max_{x\in X} \|a(x)\|.\]

Let $G$ be a locally compact group.
Throughout the paper, we always equip $G$ with its fixed left Haar measure $m$.
A $G$-\Cs-algebra is a \Cs-algebra equipped
with a pointwise norm continuous $G$-action.
A $G$-\Ws-algebra is a von Neumann algebra equipped
with a pointwise weak-$\ast$ continuous $G$-action.
A $G$-Hilbert space is a Hilbert space
equipped with a strongly continuous unitary representation of $G$.
For a $G$-\Cs-algebra $A$ and a $G$-\Ws-algebra $M$, denote by $A\rc G$, $M\vnc G$ their reduced crossed product respectively.
For $s\in G$,
denote by $u_s$ the canonical implementing unitary elements of $s$ in $\mathcal{M}(A \rc G)$ and $M\vnc G$.
To avoid confusion,
throughout the paper, we denote the product operations of the
twisted convolution algebras and the reduced crossed products
by the symbol `$\ast$'. The usual product convention is reserved for the pointwise product.

\subsection{Haagerup--Kraus's approximation property}
Let $G$ be a locally compact group.
Here we briefly review some definitions and facts related to the AP.
For details, we refer the reader to \cite{Ey} and the introduction of \cite{HK}.

\subsection*{Fourier algebra $A(G)$}
Define $A_c(G)$ to be the space of compactly supported functions $\varphi$
of the form $\varphi(s)= \langle s.\xi, \eta\rangle$
for some $\xi$, $\eta$ in a $G$-Hilbert space.
For $\varphi \in A_c(G)$, define
$\| \varphi \|_{A(G)}$ to be the infimum of
$\|\xi \|\|\eta\|$, where $\xi$, $\eta$ run over all vectors satisfying the above equation.
The Fourier algebra $A(G)$ of $G$ is the completion
of the normed space $(A_c(G), \|\cdot \|_{A(G)})$.
There is a canonical identification $A(G)= L(G)_\ast$.
\subsection*{Completely bounded multipliers $M_0A(G)$}
A function $\varphi \colon G \rightarrow \mathbb{C}$ is said to be a \emph{multiplier} of $A(G)$
if it satisfies $\varphi \cdot A(G) \subset A(G)$.
Since $A(G)^\ast= L(G)$, the multiplication by a multiplier $\varphi$ induces a normal operator ${\rm M}_\varphi$ on $L(G)$.
Define $M_0A(G)$ to be the set of all multipliers $\varphi$ of $A(G)$
whose induced map ${\rm M}_\varphi$ is completely bounded.
Note that for any $\varphi \in M_0A(G)$ and $a\in C_c(G)$,
we have ${\rm M}_\varphi(a)=\varphi \cdot a$.
By this formula, each $\varphi \in M_0A(G)$
defines a completely bounded operator $\bar{{\rm M}}_\varphi:= {\rm M}_\varphi|_{\rg(G)}$ on $\rg(G)$.
We equip $M_0A(G)$ with the norm
$\|\varphi\|_{M_0A(G)}:= \|{\rm M}_\varphi\|_{\rm cb}= \|\bar{{\rm M}}_{\varphi}\|_{\rm cb}$.
With this norm, $M_0A(G)$ forms a Banach space.
Note that $A(G) \subset M_0A(G) \subset C_b(G)$.
\subsection*{Predual $Q(G)$ of $M_0A(G)$}
Each $f\in L^1(G)$ defines an element of $M_0A(G)^\ast$
by \[\langle \varphi, f \rangle := \int_G \varphi(s)f(s)dm(s);~ \varphi\in M_0A(G).\]
The norm closure of $L^1(G)$ in $M_0A(G)^\ast$ is denoted by $Q(G)$.
The canonical pairing gives the identification $Q(G)^\ast = M_0A(G)$.
See Proposition 1.10 (b) in \cite{dCH} for a proof.

Now we are able to state the definition of the AP \cite{HK}.
\begin{Def}\label{Def:AP}
A locally compact group $G$ is said to have the {\it approximation property} (AP)
if there is a net $(\varphi_i)_{i\in I}$ in $A(G)$
converging to the constant function $1$ in the $\sigma(M_0A(G), Q(G))$-topology.
\end{Def}
As mentioned in Remark 1.2 of \cite{HK},
when $G$ has the AP,
one can choose an approximation net $(\varphi_i)_{i\in I}$ in the definition
from $A_c(G)$.

Recall that $G$ is {\it weakly amenable}
(\cite{Ha79c}, \cite{dCH}, \cite{CH})
if and only if the net $(\varphi_i)_{i \in I}$ in Definition \ref{Def:AP}
can be chosen to be $\| \cdot \|_{M_0A(G)}$-bounded.
In general weak amenability is much stronger than the AP.
An advantage of the AP is that it is stable under taking
various operations including extensions and free products \cite{HK}, \cite{BO}.
In contrast to the AP, weak amenability is easily broken by taking extensions (see \cite{Ha16}, \cite{OP}, \cite{Oz12}), and it is not known
if weak amenability is stable under taking free products.

The AP is useful to analyze the position of an element in the tensor products and reduced crossed products; see \cite{HK}, \cite{Suz17}, \cite{Suzmin}, \cite{Suz18}, \cite{Zac}, \cite{Zac2}, \cite{Zsi} for some applications.
\subsection{Operator approximation properties and Fubini-type theorems}
Since the definition of the AP does not require boundedness,
its operator algebraic counterparts naturally involve unbounded approximations.
To control behavior of unbounded approximations,
it is natural and convenient to introduce the following topologies
on the space of completely bounded operators.
Here we recall their definitions.

\subsection*{Stable pointwise topologies}
For \Cs-algebras $A, B$, denote by ${\rm CB}(A, B)$ the space of completely bounded linear maps from $A$ to $B$.
A net $(\varphi_i)_{i\in I}$ in ${\rm CB}(A, B)$ is said to converge to $\varphi\in {\rm CB}(A, B)$
in the {\it stable point-norm topology} if
$\id_{\mathbb{K}} \otimes \varphi_i$ converges to $\id_{\mathbb{K}} \otimes \varphi$ in the pointwise norm topology.
A net $(\varphi_i)_{i\in I}$ in ${\rm CB}(A, B)$ is said to converge to $\varphi\in {\rm CB}(A, B)$
in the {\it strong stable point-norm topology} if
$\id_{\mathbb{B}} \otimes \varphi_i$ converges to $\id_{\mathbb{B}} \otimes \varphi$ in the pointwise norm topology.
A \Cs-algebra $A$ is said to have the {\it operator approximation property} (OAP)
(resp.~the {\it strong OAP} (SOAP))
if $\id_A$ is in the closure of finite rank operators in ${\rm CB}(A):={\rm CB}(A, A)$
in the stable (resp.~the strong stable) point-norm topology.
Similarly, for von Neumann algebras $N, M$,
denote by ${\rm CB}_n(N, M)$ the space of normal completely bounded linear maps from $N$ to $M$.
A net $(\varphi_i)_{i\in I}$ in ${\rm CB}_n(N, M)$ is said to converge to $\varphi\in {\rm CB}_n(N, M)$
in the {\it stable point-ultraweak topology} if
$\id_{\mathbb{B}} \votimes \varphi_i$ converges to $\id_{\mathbb{B}} \votimes \varphi$ in the pointwise ultraweak topology.
A von Neumann algebra $M$ is said to have the {\it weak-$\ast$ OAP} (\mbox{W$^\ast$}OAP) if $\id_M$ is in the closure of finite rank operators in ${\rm CB}_n(M) := {\rm CB}_n(M, M)$
in the stable point-ultraweak topology.

These approximation properties are strongly related
to the {\it slice map property} \cite{Kra}.
We next review the slice map property and its connection with the OAPs.
\subsection*{Slice map property}
For \Cs-algebras $A, B$ and a closed subspace $X$ of $B$,
define
\[F(A, B, X):= \{x\in A \otimes B: (\omega \otimes\id_B)(x)\in X {\rm~for~all~}\omega \in A^\ast\}.\]
It is clear from the definition that $A \otimes X$, the norm closure of $A\odot X$ in $A\otimes B$, is contained in $F(A, B, X)$.
A triplet $(A, B, X)$ is said to have the {\it slice map property}
if they satisfy $F(A, B, X)= A\otimes X$.
The validity of the equality is also called a Fubini-type theorem.
This property is useful to study \Cs-subalgebras of tensor products
(see e.g., \cite{Zac}, \cite{Zsi}).

There is a natural von Neumann algebra analogue of the slice map property,
called the {\it weak slice map property}.
See Section 12.4 of \cite{BO} for details.
It is worth stating that Tomita's tensor commutant theorem
is equivalent to the validity of the weak slice map property for all triplets of von
Neumann algebras. See below Remark 12.4.5 of \cite{BO}.

Kraus shows that
the OAPs are characterized by slice map properties.
\begin{Thm}[\cite{Kra}, see also Theorem 12.4.4 of \cite{BO}]\label{Thm:OAP}
The following statements hold true.
\begin{enumerate}[\upshape(1)]
\item
A \Cs-algebra $A$ has the OAP if and only
if the triplet $(A, \mathbb{K}, X)$ has the slice map property for all closed subspaces $X$ of $\mathbb{K}$.
\item A \Cs-algebra $A$ has the SOAP if and only
if the triplet $(A, B, X)$ has the slice map property for all \Cs-algebras $B$ and their closed subspaces $X$.
\item A von Neumann algebra $M$ has the \mbox{W$^\ast$}OAP if and only
if the triplet $(M, N, X)$ has the weak slice map property
for all von Neumann algebras $N$ and their ultraweakly closed subspaces $X$.
\end{enumerate}
\end{Thm}
By the statement (2), the SOAP implies exactness.

Haagerup--Kraus \cite{HK} characterize the AP of a discrete group as follows.
\begin{Thm}[\cite{HK}, Theorem 2.1]
For a discrete group $\Gamma$, the following conditions are equivalent:
\begin{enumerate}[\upshape(1)]
\item $\Gamma$ has the AP,
\item $\rg(\Gamma)$ has the OAP,
\item $\rg(\Gamma)$ has the SOAP,
\item $L(\Gamma)$ has the W$^\ast$OAP.
\end{enumerate}
\end{Thm}
For general locally compact groups,
the implication ``(1) $\Rightarrow$ (4)'' is shown in \cite{HK}.
The implications ``(2), (3), (4) $\Rightarrow$ (1)''
fail in general; the group $G:=SL(3, \mathbb{R})$ does not have the AP \cite{LS},
while $\rg(G)$ and $L(G)$ are amenable (cf.~Remark 2.5 of \cite{HK}).
In this paper, we prove the new implications ``(1) $\Rightarrow$ (2), (3)''.
To the author's knowledge, this was left open for a long time.
The implications follow from a more general result on the reduced crossed products (Theorem \ref{Thm:SOAP}).

\section{Main Theorems}
One of the main ideas of the present paper is extending multiplier operators to the reduced crossed product algebras.
We then study their convergence conditions.
This is already appeared in the discrete case in the proof of Proposition 3.4 in \cite{Suz17}.

Throughout this section, let $G$ be a locally compact group.
\begin{Lem}\label{Lem:multiplier} Let $\varphi \in M_0A(G)$.
Then the following statements hold true.
\begin{enumerate}[\upshape(1)]
\item Let $A$ be a nonzero $G$-\Cs-algebra.
Then there is a completely bounded linear map
\[\bar{{\rm M}}_{A, \varphi} \colon A \rc G \rightarrow A \rc G\]
satisfying
\[\bar{{\rm M}}_{A, \varphi}(a)= \varphi \cdot a \quad {\rm for~ all~} a\in C_c(G, A).\]
Moreover we have $\|\bar{{\rm M}}_{A, \varphi}\|_{\rm cb}=\|\varphi\|_{M_0A(G)}$.
\item
Let $M$ be a nonzero $G$-\Ws-algebra.
Then there is a normal completely bounded linear map
\[{\rm M}_{M, \varphi} \colon M \vnc G \rightarrow M \vnc G\]
satisfying 
\[{\rm M}_{M, \varphi}(a)= \varphi \cdot a \quad {\rm  for~ all~} a\in C_c(G, M).\]
Moreover we have $\|{\rm M}_{M, \varphi}\|_{\rm cb}=\|\varphi\|_{M_0A(G)}$.
\end{enumerate}
\end{Lem}
\begin{proof}
\noindent (2): For the \Ws-algebra case, the construction is essentially the same
as the discrete case (cf.~ Proposition 3.4 of \cite{Suz17}).
Fix a normal faithful $\ast$-representation $M \vnc G \subset \mathbb{B}(\mathfrak{H})$.
Consider the covariant representation of $M$ on $\mathfrak{H} \otimes L^2(G)$ given by
\[x\mapsto x \otimes \id_{L^2(G)},~ s \mapsto u_s \otimes \lambda_s \quad {\rm for~}x\in M,~ s\in G.\]
Here $\lambda$ denotes the left regular representation of $G$.
Then, by Fell's absorption principal, this covariant representation is unitary equivalent to
a faithful regular covariant representation.
Hence it extends to an injective normal $\ast$-homomorphism
\[\delta \colon M \vnc G \rightarrow (M \vnc G) \votimes L(G).\]
Observe that
$\id_{M \vnc G} \votimes {\rm M}_\varphi$ preserves $\im(\delta)$.
Denote by $\gamma \colon \im(\delta) \rightarrow M \vnc G$ the inverse map of $\delta$.
Define \[{\rm M}_{M, \varphi}:=\gamma \circ (\id_{M \vnc G} \votimes {\rm M}_\varphi) \circ \delta.\]
We show that ${\rm M}_{M, \varphi}$ possesses the desired properties.
It is clear from the definition that ${\rm M}_{M, \varphi}$ is normal
and satisfies $\|{\rm M}_{M, \varphi}\|_{\rm cb}=\|\varphi\|_{M_0A(G)}$.
Since $\varphi \in C_b(G)$, it is easy to check that
${\rm M}_{\varphi}(\lambda_s)= \varphi(s)\lambda_s$ for all $s\in G$.
Therefore ${\rm M}_{M, \varphi}(xu_s)= \varphi(s)xu_s$ for all $x\in M$ and $s\in G$.
As ${\rm M}_{M, \varphi}$ is normal,
this yields
${\rm M}_{M, \varphi}(a)= \varphi \cdot a$ for all $a \in C_c(G, M)$.

\noindent (1): We first take a $G$-equivariant embedding of $A$ into a $G$-\Ws-algebra $M$.
(For instance, a faithful regular covariant representation gives such an embedding.)
The inclusion $A\subset M$ induces $A \rc G \subset M \vnc G$.
Since ${\rm M}_{M, \varphi}(C_c(G, A))=\varphi \cdot C_c(G, A)\subset C_c(G, A)$,
the restriction map ${\rm M}_{M, \varphi}|_{A \rc G}$ provides the desired map.
\end{proof}
The following lemma and its corollary are key observations.
Note that even for the trivial group actions, they generalize Propositions 1.3, 1.4,
and Theorem 1.9 of \cite{HK} in the \Cs-algebra case.
\begin{Lem}\label{Lem:Q(G)}
\begin{enumerate}[\upshape(1)]
\item
Let $A$ be a $G$-\Cs-algebra.
Let $a\in A \rc G$, $f\in (A \rc G)^\ast$.
Then the linear functional $\omega_{a, f}$ on $M_0A(G)$ defined by
\[\omega_{a, f}(\varphi):= \langle \bar{{\rm M}}_{A, \varphi}(a), f \rangle;\quad \varphi \in M_0A(G)\]
is contained in $Q(G)$.
\item Let $M$ be a $G$-\Ws-algebra.
Let $a \in M \vnc G$, $f\in (M \vnc G)_\ast$, $\psi \in A_c(G)$.
Then the linear functional $\omega_{a, f, \psi}$ on $M_0A(G)$ defined by
\[\omega_{a, f, \psi}(\varphi):= \langle {\rm M}_{M, \psi \ast \varphi}( a), f \rangle;\quad \varphi\in M_0A(G)\]
 is contained in $Q(G)$.
\end{enumerate}
\end{Lem}
\begin{proof}
(1):~Lemma \ref{Lem:multiplier} implies $\|\omega_{a, f}\|_{M_0A(G)^\ast} \leq \|a\| \|f\|$
for all $a\in A \rc G$ and $f\in (A \rc G)^\ast$.
Hence it suffices to show the claim
in the case
$a\in C_c(G, A)$. Fix $a\in C_c(G, A)$ and $f\in (A\rc G)^\ast$.
Let $f^{\mathcal{M}}$ be the extension of $f$ to $\mathcal{M}(A \rc G)$ which is strictly continuous on the unit ball of $\mathcal{M}(A \rc G)$.
(The restriction of the double dual map $f^{\ast\ast} \colon (A \rc G)^{\ast\ast}  \rightarrow \mathbb{C}$ to $\mathcal{M}(A \rc G) \subset (A \rc G)^{\ast\ast}$ gives the desired extension.)
Define the map $\tilde{f} \colon G \rightarrow A^\ast$
to be $\langle x, \tilde{f} (s)\rangle:= \langle x u_s, f^{\mathcal{M}}\rangle$ for $x\in A$, $s\in G$.
Then $\tilde{f}$ is bounded and weak-$\ast$ continuous.
Observe that, with respect to the strict topology, we have
\[\bar{{\rm M}}_{A, \varphi}(a)= \int_{G} \varphi(s)a(s)u_s dm(s).\]
By the strict continuity of $f^{\mathcal{M}}$, for any $\varphi \in A(G)$,
\[\omega_{a, f}(\varphi) = \int_G \langle \varphi(s)a(s)u_s, f^{\mathcal{M}}\rangle dm(s)= \int_G \varphi(s)\langle a(s), \tilde{f} (s)\rangle dm(s).\]
Since $a\in C_c(G, A)$ and $\tilde{f}$ is bounded and weak-$\ast$ continuous,
we conclude \[\omega_{a, f} \in C_c(G) \subset Q(G).\]

\noindent
(2):~ We fix a faithful normal $\ast$-representation
$M \vnc G \subset \mathbb{B}(\mathfrak{H})$.
We then identify $M \vnc G$ with a von Neumann subalgebra of $\mathbb{B}(\mathfrak{H})  \votimes L(G)$ via the normal embedding
\[\delta \colon M \vnc G \rightarrow (M \vnc G) \votimes L(G) \subset \mathbb{B}(\mathfrak{H})  \votimes L(G)\] defined in the proof of Lemma \ref{Lem:multiplier} (2).
We extend $f$ to a normal linear functional $g$ on $\mathbb{B}(\mathfrak{H})   \votimes L(G)$. 
Then, since $\delta \circ {\rm M}_{M, \phi}= (\id_{M \vnc G} \votimes {\rm M}_\phi) \circ \delta$ for all $\phi \in M_0A(G)$ (by the definition of ${\rm M}_{M, \phi}$),
we obtain
\[\omega_{a, f, \psi}(\varphi)=\langle (\delta \circ {\rm M}_{M, \psi \ast \varphi})(a), g \rangle = \langle (\id_{\mathbb{B}(\mathfrak{H})} \votimes {\rm M}_{\psi \ast \varphi})(\delta(a)), g\rangle \quad {\rm for~all~}\varphi \in M_0A(G).\]
It is proved in Proposition 1.3 (a) of \cite{HK}
that linear functionals of this form sit in $Q(G)$.
\end{proof}
\begin{Cor}\label{Cor:converge}
Let $G$ be a locally compact group with the AP.
\begin{enumerate}[\upshape(1)]
\item
Let $A$ be a $G$-\Cs-algebra.
Then there is a net $(\varphi_i)_{i\in I}$ in $A_c(G)$ with the following property:
the net $\bar{{\rm M}}_{A, \varphi_i}$ converges to $\id_{A \rc G}$ in the strong stable
point-norm topology. 
\item
Let $M$ be a $G$-\Ws-algebra.
Then there is a net $(\varphi_i)_{i\in I}$ in $A_c(G)$ with the following property:
the net ${\rm M}_{M, \varphi_i}$ converges to $\id_{M \vnc G}$ in the stable
point-ultraweak topology.
\end{enumerate}
\end{Cor}
\begin{proof}
(1):~Since $G$ has the AP, one can choose a net $(\varphi_i)_{i\in I}$ in $A_c(G)$
converging to $1$ in the $\sigma(M_0A(G), Q(G))$-topology.
By applying Lemma \ref{Lem:Q(G)} (1) to $\mathbb{B} \otimes A$, where $\mathbb{B}$ is equipped with the trivial $G$-action, we conclude that
the net $\id_{\mathbb{B}} \otimes \bar{{\rm M}}_{A, \varphi_i}$ pointwise weakly converges to $\id_{\mathbb{B}}\otimes \id_{A \rc G}$.
Now a standard application of the Hahn--Banach theorem
implies the existence of the desired net in the convex hull of $\{\varphi_i: i \in I\}$.

\noindent
(2):~ Take a non-negative function $\psi$ in $A_c(G)$ with $\int_G \psi dm =1$.
Then observe that $\psi \ast 1 =1$.
Take a net $(\varphi_i)_{i\in I}$ in $A_c(G)$ converging
to $1$ in the $\sigma(M_0A(G), Q(G))$-topology.
By Lemma \ref{Lem:Q(G)} (2) (applied to $\mathbb{B}\votimes M$), the net $(\psi \ast \varphi_i)_{i\in I}$
has the desired property.
\end{proof}
The next lemma is important to carry out mollification arguments.
For a locally compact group $G$, denote by $\Delta_G$ the modular function of $G$.
That is, the continuous multiplicative function
$\Delta_G \colon G \rightarrow \mathbb{R}^\times_+$
satisfying
$m(Us)=\Delta_G(s)m(U)$ for all $s\in G$, $U\subset G$.

\begin{Lem}\label{Lem:ineq}
Let $A$ be a $G$-\Cs-algebra.
Then for any $a\in C_c(G, A)$, $b, c\in C_c(G)$, and $s\in G$,
we have $\|(b^\ast \ast a \ast c)(s)\|_A \leq \Delta_G(s)^{-1/2}\|b \|_2 \|c\|_2\|a\|_{A\rc G} $.
\end{Lem}
\begin{proof}
Denote by $\tilde{\rho} \colon G \curvearrowright C_c(G)$
the (non-normalized) right translation action
\[\tilde{\rho}_v(d)(w):=d(wv),\quad d\in C_c(G), v, w\in G.\]
Then direct computations show
\[(b^\ast \ast a \ast c)(s)= (b^\ast \ast a \ast \tilde{\rho}_s(c))(e),\qquad \|\tilde{\rho}_s(c)\|_2 = \Delta_G(s)^{-1/2}\|c\|_2.\]
Hence it suffices to show the statement for $s=e$.

Fix a non-degenerate faithful $\ast$-representation $A\rc G\subset \mathbb{B}(\mathfrak{H})$.
Define a covariant representation of $A$ on $L^2(G)\otimes \mathfrak{H} = L^2(G, \mathfrak{H})$
as follows.
\[x\mapsto  \id_{L^2(G)} \otimes x,~
s\mapsto \lambda_s\otimes u_s;~x\in A,~s\in G.\]
By Fell's absorption principal, this covariant representation
induces a faithful $\ast$-representation
\[\sigma \colon A \rc G \rightarrow \mathbb{B}(L^2(G, \mathfrak{H})).\]
Denote its strictly continuous extension on $\mathcal{M}(A\rc G)$ by the same symbol $\sigma$.

We next define $V, W \in \mathbb{B}(\mathfrak{H}, L^2(G, \mathfrak{H}))$
to be
\[[V(\xi)](s):=c(s)u_s\xi,\quad [W(\xi)](s):= b(s)u_s\xi \quad {\rm for~} \xi \in \mathfrak{H},~s\in G.\]
Note that $\|V\| \leq \|c \|_2$, $\|W\| \leq \|b\|_2$.
We show
$(b^\ast \ast a \ast c)(e)= W^\ast\sigma(a) V$.
By the above inequalities, this completes the proof.
To prove the claim,
it suffices to show the equation for $xu_s \in \mathcal{M}(A\rc G)$ instead of $a$ for each $x\in A$ and $s\in G$.
For $x\in A$, $s\in G$, $\xi$, $\eta\in \mathfrak{H}$,
\begin{eqnarray*} \langle \sigma(x u_s)V\xi, W\eta\rangle
&=&\int_G \langle x u_s c(s^{-1}t)u_{s^{-1}t}\xi, b(t)u_t\eta\rangle dm(t)\\
&=&\int_G \langle \overline{b(t)} \alpha_{t^{-1}}(x)c(s^{-1}t)\xi, \eta\rangle dm(t)\\
&=&\int_G \langle \overline{b(t^{-1})} \Delta_G(t^{-1})\alpha_{t}(x) c(s^{-1}t^{-1})\xi, \eta\rangle dm(t) \\
&=& \langle [(b^\ast \ast xu_s \ast c)(e)]\xi, \eta \rangle.
\end{eqnarray*}
This proves $(b^\ast \ast xu_s \ast c)(e)= W^\ast\sigma(xu_s) V$.
\end{proof}

We now obtain another key result.
\begin{Prop}\label{Prop:approx}
Let $A$ be a \Cs-algebra with the SOAP $($resp.~OAP$)$.
Then for any $\varphi\in A_c(G)$,
the map $\bar{{\rm M}}_{A, \varphi}$ is obtained as the limit of finite rank operators
in ${\rm CB}(A \rc G)$ in the strong stable point-norm topology $($resp.~ in the stable point-norm topology$)$.
\end{Prop}
\begin{proof}
We only show the statement for the SOAP. The OAP case is essentially the same.

Denote by $\mathcal{F} \subset {\rm CB}(A \rc G)$ the closure of the set of all finite rank operators in ${\rm CB}(A \rc G)$ with respect to the strong stable point-norm topology.
Let $\mathcal{N}$ be the directed set of all compact neighborhoods of $e\in G$.
For each $U\in \mathcal{N}$,
put 
\[\theta_U := m(U)^{-2} \chi_U \ast \chi_U \in C_c(G) \subset \mathcal{M}(A\rc G).\]
Note that $\|\theta_U\|_{\mathcal{M}(A \rc G)} \leq \|\theta_U\|_1 =1$.
From this, it is easy to show
\[\lim_{U\in \mathcal{N}} \theta_U^\ast \ast a \ast \theta_U =a\]
for all $a\in A \rc G$.
Therefore it suffices to show that the map
\[\bar{{\rm M}}_{A, \varphi, U} := \bar{{\rm M}}_{A, \varphi}(\theta_U^\ast \ast \cdot \ast \theta_U)\]
is contained in $\mathcal{F}$ for all $U \in \mathcal{N}$.

Put $V(\varphi):=\varphi^{-1}(\mathbb{C}\setminus \{0\})$.
Note that $V(\varphi)$ is relatively compact and open in $G$.
We first show that $\im(\bar{{\rm M}}_{A, \varphi, U})\subset
C_0(V(\varphi), A)$ ($\subset A\rc G$).
Let $a\in C_c(G, A)$ be given.
Then
\[\bar{{\rm M}}_{A, \varphi, U}(a) =\varphi \cdot (\theta_U^\ast \ast a \ast \theta_U) \in C_0(V(\varphi), A).\]
This implies
\[\|\bar{{\rm M}}_{A, \varphi, U}(a)\|_{A \rc G} \leq\|\bar{{\rm M}}_{A, \varphi, U}(a)\|_1 \leq
m(V(\varphi))\|\bar{{\rm M}}_{A, \varphi, U}(a)\|_\infty.\]
We will estimate the norm $\|\bar{{\rm M}}_{A, \varphi, U}(a)\|_\infty$.
Put
\[C_{\varphi}:= \max \{\Delta_G(s)^{-1/2} : s\in \supp(\varphi)\}.\]
Then Lemma \ref{Lem:ineq} implies
\[\|\bar{{\rm M}}_{A, \varphi, U}(a)\|_\infty \leq \|\varphi\|_\infty
\| (\theta_U^\ast \ast a \ast \theta_U)|_{V(\varphi)}\|_{\infty}
\leq C_{\varphi} \|\varphi\|_\infty \|\theta_U\|_2^2\|a\|_{A\rc G}.\]
These inequalities yield
\[\im(\bar{{\rm M}}_{A, \varphi, U})\subset
C_0(V(\varphi), A) ~(\subset A\rc G).\]

We next consider the map $\Phi_{A, \varphi, U} \colon A \rc G \rightarrow C_0(V(\varphi)) \otimes A$ obtained by composing
$\bar{{\rm M}}_{A, \varphi, U}$ with the canonical linear
isomorphism $C_0(V(\varphi), A) \cong C_0(V(\varphi)) \otimes A$.
Observe that $\Phi_{A, \varphi, U}$ is bounded by the inequalities in the previous paragraph. 
We also observe that
the canonical inclusion map
$\iota_{V(\varphi)} \colon C_0(V(\varphi))\otimes A \cong C_0(V(\varphi), A) \rightarrow A \rc G$ is bounded.
By passing to the stabilizations, we obtain complete boundedness of
$\Phi_{A, \varphi, U}$ and $\iota_{V(\varphi)}$.

Obviously $\bar{{\rm M}}_{A, \varphi, U}=\iota_{V(\varphi)} \circ \Phi_{A, \varphi, U}$.
Thus $\bar{{\rm M}}_{A, \varphi, U}$ factors through
$C_0(V(\varphi))\otimes A$ in the category of operator spaces.
Since $C_0(V(\varphi)) \otimes A$ has the SOAP, we have $\bar{{\rm M}}_{A, \varphi, U}\in \mathcal{F}$.
\end{proof}
Now we are able to prove the \Cs-algebra analogues of 
Theorem 3.1 (a) and Theorem 3.2 (a), (b) in \cite{HK}.
We point out that their proofs involve unbounded operator-valued weights 
(\cite{Ha78a}, \cite{Ha78b}, \cite{Ha79a}, \cite{Ha79b}).
Since there are no corresponding results for \Cs-algebras,
their proofs do not seem to work in the \Cs-algebra case.
\begin{Thm}\label{Thm:SOAP}
Let $G$ be a locally compact group with the AP.
Let $A$ be a $G$-\Cs-algebra.
Then the reduced crossed product $A\rc G$ has the SOAP
if and only if the \Cs-algebra $A$ has the SOAP.
In particular the reduced group \Cs-algebra $\rg(G)$ has the SOAP.
The analogous statement also holds true for the OAP.
\end{Thm}
\begin{proof}
We only show the statement for the SOAP. The OAP case is essentially the same.

When $A$ has the SOAP, so does $A\rc G$ by Corollary \ref{Cor:converge}
and Proposition \ref{Prop:approx}.

We show the converse for general locally compact groups $G$.
Assume that $A$ does not have the SOAP.
By Theorem \ref{Thm:OAP}, there is a \Cs-algebra $B$ and a closed subspace $X$
satisfying $F(A, B, X) \neq A \otimes X$.
We will show that the triplet $(A \rc G, B, X)$ does not have
the slice map property.
Take a non-negative function $a\in C_c(G)$ with $a(e)=1$.
Take $x\in F(A, B, X) \setminus (A \otimes X)$.
Define $b\in C_c(G, A\otimes B)$ to be $b(s):= a(s)x$; $s\in G$.
We first show that $b\in F(A\rc G, B, X)$.
As we have seen in the proof of Lemma \ref{Lem:Q(G)}, for any $\omega\in (A\rc G)^\ast$,
there is a bounded weak-$\ast$ continuous map $\tilde{\omega}\colon G \rightarrow A^\ast$ satisfying
\[(\omega \otimes \id_B)(b)= \int_G a(s) (\tilde{\omega}(s) \otimes \id_B)(x)dm(s).\]
Since the integrants are norm continuous and take their values in $X$,
we have $b \in F(A\rc G, B, X)$.
We next show $b \not\in (A \rc G) \otimes X$.
Let $\mathcal{N}$ be the directed set of all compact neighborhoods of $e\in G$.
Take $\varphi \in A_c(G)$
satisfying $\varphi(e)=1$.
Then, on the one hand, for any $U\in \mathcal{N}$, we have
\[(\Phi_{A, \varphi, U} \otimes \id_B)([A\rc G] \otimes X) \subset
 C_0(V(\varphi))\otimes A \otimes X.\]
Here and below we adopt the notations in the proof of Proposition \ref{Prop:approx}.
On the other hand,
since $b\in C_c(G, A\otimes B)$,
we have
\[\lim_{U\in \mathcal{N}} [(\Phi_{A, \varphi, U} \otimes \id_B)(b)](e) = \lim_{U\in \mathcal{N}} (\theta_U^\ast \ast b \ast \theta_U)(e)=  x \not\in A \otimes X.\]
Since $A \otimes X$ is closed in $A \otimes B$, this yields that, for some $U\in \mathcal{N}$,
\[[(\Phi_{A, \varphi, U} \otimes \id_B)(b)](e) \not \in A \otimes X.\]
This concludes $b \not \in  (A\rc G) \otimes X$.
By Theorem \ref{Thm:OAP} (2), $A \rc G$ cannot have the SOAP.
\end{proof}

As another consequence of our method, we establish the implication ``AP $\Rightarrow$ exactness'' for general locally compact groups.
This strengthens Corollary E of \cite{BCL},
where this implication is shown for second countable weakly amenable groups.
However we emphasize that weak amenability is much stronger
than the AP.
Recall from \cite{KW} that a locally compact group $G$ is said to be {\it exact}
if the functor $- \rc G$ is exact.
Note that not all groups are exact \cite{Gro}, \cite{Osa},
and there are exact groups without the AP \cite{LS}, \cite{HT}.
We remark that it is not known in general if the exactness of
a locally compact group follows from the exactness of the reduced group \Cs-algebra.
Thus, unlike the discrete case,
we cannot deduce the following theorem from Theorem \ref{Thm:SOAP}.
\begin{Thm}\label{Thm:exact}
All locally compact groups with the AP are exact.
\end{Thm}

\begin{proof}
Let $G$ be a locally compact group with the AP.
Let $A$ be a $G$-\Cs-algebra and $I$ be a closed $G$-invariant ideal of $A$.
By Corollary \ref{Cor:converge},
one can take a net $(\varphi_i)_{i\in I}$ in $A_c(G)$
such that the net
$(\bar{{\rm M}}_{A, \varphi_i})_{i\in I}$ converges to $\id_{A \rc G}$ in the pointwise norm topology.
Let $\mathcal{N}$ denote the directed set of all compact neighborhoods of $e$ in $G$.
For $U \in \mathcal{N}$,
put $\theta_U:=m(U)^{-2}\chi_U \ast \chi_U$.
For $i\in I$ and $U \in \mathcal{N}$,
define \[\bar{{\rm M}}_{A, \varphi_i, U}:=\bar{{\rm M}}_{A, \varphi_i}(\theta_U^\ast \ast \cdot \ast \theta_U),~ \bar{{\rm M}}_{A/I, \varphi_i, U}:=\bar{{\rm M}}_{A/I, \varphi_i}(\theta_U^\ast \ast \cdot \ast \theta_U)\]
 as in the proof of Proposition \ref{Prop:approx}.
Let
$\pi \colon A \rc G \rightarrow (A/I) \rc G$ denote the canonical quotient map.
Then direct computations show that
\[\pi \circ \bar{{\rm M}}_{A, \varphi_i, U}= \bar{{\rm M}}_{A/I, \varphi_i, U} \circ \pi\] for all $i\in I$ and $U \in \mathcal{N}$.
(Indeed, by the boundedness of both sides, it suffices to check the equality on $C_c(G, A)$, which is obvious.)
This shows 
$\bar{{\rm M}}_{A, \varphi_i, U}(\ker(\pi)) \subset \ker(\pi).$
It is shown in the proof of Proposition \ref{Prop:approx} that
$\im(\bar{{\rm M}}_{A, \varphi_i, U}) \subset C_c(G, A)$.
Therefore
\[\bar{{\rm M}}_{A, \varphi_i, U}(\ker(\pi)) \subset \ker(\pi) \cap C_c(G, A)= C_c(G, I).\]
Since $x =\lim_{U \in \mathcal{N}} \lim_{i\in I} \bar{{\rm M}}_{A, \varphi_i, U}(x)$ for all $x\in A \rc G$,
we conclude $\ker(\pi)=I \rc G$. 
\end{proof}
\begin{Rem}
Our proofs also work in the \Ws-case after slight modifications.
For instance, the map $\Phi$ extends to a normal map
as it can be given spatially (see Lemma \ref{Lem:ineq}).
We thus obtain alternative proofs of Theorem 3.1 (a) and Theorem 3.2 (a), (b) of \cite{HK}. 
\end{Rem}
\subsection*{Acknowledgements}
The author would like to thank Tim de Laat for
pointing out an inaccuracy in Section 2.2 in the previous version.
We also thank a referee (of another journal) for helpful comments.
Finally we thank the referee for valuable comments.
This work was supported by JSPS KAKENHI Grant-in-Aid for Young Scientists
(Start-up, No.~17H06737) and tenure track funds of Nagoya University.

\end{document}